\documentclass[oneside,reqno,10pt]{amsart}
\usepackage{amsfonts,amssymb,amsmath,amsthm,color}

\usepackage{hyperref}

\usepackage{float}
\floatstyle{plaintop}
\restylefloat{table}
\usepackage[tableposition=top]{caption}
\usepackage{color}
\usepackage{hyperref}
\usepackage{url}
\usepackage{breakurl}
\newcommand{\bburl}[1]{\textcolor{blue}{\url{#1}}}

\theoremstyle{plain} 
\newtheorem{theorem}             {Theorem}  [section]

\newtheorem{lemma}{Lemma}[section]

\theoremstyle{definition}
\newtheorem{conj}{Conjecture}
\newtheorem{rek}{Remark}
\newtheorem{defn}{Definition}[section]

\theoremstyle{remark}

\newcommand{\seqnum}[1]{\href{https://oeis.org/#1}{\rm \underline{#1}}}

\def\modd#1 #2{#1\ \mbox{\rm (mod}\ #2\mbox{\rm )}}

\numberwithin{equation}{section}

\begin{document}

\author{Muhammet Boran}
\address{Department of Mathematics, Yıldız Technical University, 34220 Esenler, Istanbul, TURKEY}
\email{\bburl{muhammet.boran@std.yildiz.edu.tr}}

\author{Garam Choi}
\address{Department of Mathematics, Colby College, Waterville, ME 04901, USA}
\email{\bburl{gchoi23@colby.edu}}

\author{Steven J.\ Miller}
\address{Department of Mathematics and Statistics, Williams College, Williamstown, MA 01267, USA}
\email{\bburl{sjm1@williams.edu}}

\author{Jesse Purice}
\address{School of Mathematical and Statistical Sciences, Arizona State University, Tempe, AZ 85287, USA}
\email{\bburl{jpurice@asu.edu}}

\author{Daniel Tsai}
\address{Department of Mathematics, National Taiwan University, No. 1, Sec. 4, Roosevelt Rd., Taipei 10617, Taiwan (R.O.C.)}
\email{\bburl{tsaidaniel@ntu.edu.tw}}

\title{A characterization of prime $v$-palindromes}

\begin{abstract}
An integer $n\geq 1$ is a $v$-palindrome if it is not a multiple of $10$, nor a decimal palindrome, and such that the sum of the prime factors and corresponding exponents larger than $1$ in the prime factorization of $n$ is equal to that of the integer formed by reversing the decimal digits of $n$. For example, if we take 198 and its reversal 891, their prime factorizations are $198 = 2\cdot 3^2\cdot 11$ and $891  =  3^4\cdot 11$ respectively, and summing the numbers appearing in each factorization both give 18. This means that $198$ and $891$ are $v$-palindromes. We establish a characterization of prime $v$-palindromes: they are precisely the larger of twin prime pairs of the form $(5 \cdot 10^m - 3, 5 \cdot 10^m - 1)$, and thus standard conjectures on the distribution of twin primes imply that there are only finitely many prime $v$-palindromes.
\end{abstract}


\thanks{This work was supported in part by the 2022 Polymath Jr REU program.}

\maketitle

\section{Introduction}

\subsection{Related Work}

There have been many papers studying properties shared by numbers and their reversals. We first set some notation.

\begin{defn}
  Let $b\geq2$, $L\geq1$, and $0\leq a_0,a_1,\ldots,a_{L-1}<b$ be any integers. We denote
  \begin{equation}
    (a_{L-1}\cdots a_1a_0)_b \ :=\  \sum^{L-1}_{i=0}a_ib^i.
  \end{equation}
  We also write $(a_{L-1},\ldots,a_1,a_0)_b$ to make it clear which are each digit.
\end{defn}

\begin{defn}
Let the base $b\geq2$ representation of an integer $n\geq1$ be\\ $(a_{L-1}\cdots a_1a_0)_b$, where $a_{L-1}\neq0$. The $b$-\emph{reverse} of $n$ is defined to be
\begin{equation}
  r_b(n) \ :=\  (a_0a_1\cdots a_{L-1})_b.
\end{equation}
We write $r(n)$ for $r_{10}(n)$.
\end{defn}
So for example $r(198) = 891$.

In \emph{A Mathematician's Apology} \cite{Har}, G.\ H.\ Hardy states that ``8712 and 9801 are the only 4-digit numbers which are integral multiples of their decimal reversal'':
\begin{equation}
    8712\ =\ 4\cdot r(8712),\ \ 9801\ = \ 9\cdot r(9801).
\end{equation}
In 1966, A.\ Sutcliffe  \cite{Su} generalized this observation and studied all integer solutions of the equation
\begin{equation}
    k\cdot r_{b}(n)\ = \ n,
\end{equation} where $b\geq2$ is the base and $0 < k < n$.  In  \cite{KS}, numbers $n$ such that $n$ divides $r(n)$ are mentioned. In particular, numbers of the form
\begin{equation}
    2178,\ \ \ 21978,\ \ \ 219978,\ \ \ 2199978,\ \ \ \dots,
\end{equation}
with any number of $9$'s in the middle, all satisfy $4n=r(n)$.

Suppose that the prime factorization of an integer $n\geq1$ is
\begin{equation}
  n\ =\ p_{1}^{\alpha_1}\cdots p_{k}^{\alpha_k},
\end{equation}
where $p_1<\cdots <p_k$ are primes and $\alpha_1,\ldots,\alpha_k\geq1$ integers. In 1977, P.\ Erd\H{o}s and K.\ Alladi  \cite{AE}  studied the function
\begin{equation}  \label{prfctr}
  A(n)\ =\ \sum_{i=1}^{k} p_i \cdot \alpha_i.
\end{equation}
The entries A008474 and A000026 from the OEIS \cite{OEIS} are
\begin{align}
  F(n) &\ =\ \sum_{i=1}^{k} \left(p_i +\alpha_i\right), \\
  G(n)  & \ =\ \prod_{i=1}^k p_i \cdot \alpha_i ,
\end{align}
respectively. These functions are somehow similar in expression. We introduce an arithmetic function denoted by $v(n)$ which is obtained from $F(n)$ by replacing $\alpha_i$ with 0 when $\alpha_i = 1$. In other words,
\begin{equation}\label{defofv}
  v(n) \ =\ \sum_{\substack{1\leq i\leq k,\\ \alpha_i = 1}}p_i + \sum_{\substack{1\leq i\leq k,\\ \alpha_i \geq 2}} \left(p_i +\alpha_i\right).
\end{equation}

\subsection{$v$-palindromes}

The concept of $v$-palindromes was introduced by Tsai in \cite{tsai0,tsai} and explored further in four later manuscripts \cite{tsai6,tsai3,tsai4,tsai5}. As in the abstract, consider the number $198$ whose digit reversal is $891$. Their prime factorizations are
\begin{align}
  198 &\ =\  2\cdot 3^2\cdot 11,\label{rei1}\\
  891 &\ =\  3^4\cdot 11\label{rei2},
\end{align}
and we have
\begin{equation}
  2+(3+2)+11 \ = \ (3+4) + 11.
\end{equation}
In other words, the sum of the numbers ``appearing'' on the right-hand side of \eqref{rei1} equals that of \eqref{rei2}. We first give the following definitions.

\begin{defn}
  The additive function $v\colon \mathbb{N}\to\mathbb{Z}$ is defined by setting $v(p):=p$ for primes $p$ and $v(p^\alpha):=p+\alpha$ for prime powers $p^\alpha$ with $\alpha\geq2$.
\end{defn}

Notice that this definition of $v(n)$ agrees with \eqref{defofv}. We define $v$-palindromes as follows.

\begin{defn}\label{def.4}
  Let $n\geq1$ and $b\geq2$ be integers. Then $n$ is a $v$-\emph{palindrome in base} $b$ if $b\nmid n$, $n\neq r_b(n)$, and $v(n) = v(r_b(n))$. A $v$-palindrome in base $10$ is simply called a $v$-\emph{palindrome}.
\end{defn}

Thus $198$ and $891$ are $v$-palindromes. The following are infinite sequences of $v$-palindromes \cite{tsai0,tsai}:
\begin{gather}
  18,\ 198,\ 1998,\  19998,\  199998,\  1999998,\  \dots,\label{ten1}\\
  18,\  1818,\  181818,\  18181818,\  1818181818,\  181818181818,\    \dots\label{ten2}.
\end{gather}

In \eqref{ten1}, we simply keep increasing the number of $9$'s in the middle; in \eqref{ten2}, we simply keep appending another $18$. The sequences \eqref{ten1} and \eqref{ten2} are actually parts of a larger family of $v$-palindromes derived in \cite[Theorem 3]{tsai3}. In particular, there are infinitely many $v$-palindromes. According to \cite{tsai0}, the $v$-palindromes $n\leq 10^5$ with $n<r(n)$ are
\begin{align}
  &18,\  198,\  576,\  819,\  1131,\  1304,\  1818,\  1998,\  2262,\  3393,\  4154,\  4636,\  8749,\  12441,\nonumber \\
  &14269,\ 14344,\ 15167,\ 15602,\ 16237,\ 18018,\ 18449,\ 18977,\ 19998,\ 23843,\ 24882,\nonumber\\
  &26677,\ 26892,\ 27225,\  29925,\  31229,\  36679,\  38967,\  39169,\  42788,\  45694,\  46215,\nonumber \\
  &46655,\  47259,\  48048,\  52416,\  56056,\  60147,\  62218,\  66218,\  79689,\  97999.\nonumber
\end{align}
The sequence of $v$-palindromes $n$ (whether $n<r(n)$ or not) is \seqnum{A338039} in the OEIS \cite{OEIS}. In \cite{tsai0}, it is said that extensive computer calculations suggest the following.

\begin{conj}\label{1.conj}
There are no prime $v$-palindromes.
\end{conj}

\subsection{New Results}

We are able to make significant progress towards a possible proof of Conjecture \ref{1.conj} by proving the following characterization of prime $v$-palindromes.

\begin{theorem}\label{thm:main}
  The prime $v$-palindromes are precisely the primes of the form
  \begin{equation}
    5\cdot 10^m - 1 \ =\  4\underbrace{9\cdots 9}_{m},\label{p1}
  \end{equation}
  for some integer $m\geq4$, such that
  \begin{equation}
    5\cdot 10^m-3 \ =\  4\underbrace{9\cdots 9}_{m-1}7 \label{p2}
  \end{equation}
  is also prime.
\end{theorem}

Here, the $m\geq4$ is not very significant and only means that it has been checked that there are no prime $v$-palindromes of fewer than $5$ decimal digits, and thus conceivably can be improved with more checking for small values of $m$. From this characterization of prime $v$-palindromes, it is a consequence of standard models for primes (such as the Cram\'er model, though weaker assumptions suffice) that there are only finitely many prime $v$-palindromes. In particular, we have the following.

\begin{theorem}\label{thm:heuristicfinite} Assume that the probability $n$ and $n+2$ are both prime is bounded by $C / \log^2 n$ for some $C$. Then there are only finitely many prime $v$-palindromes. \end{theorem}

The main purpose of this paper is to prove Theorem \ref{thm:main}. In Section \ref{sec:lemmas}, we give various definitions and lemmas to be used throughout the rest of the paper. The proof of the forward direction of Theorem \ref{thm:main} consists of Sections \ref{sec:setup} to \ref{sec:otherdec}. The proof of the converse consists of just Section \ref{sec:converse}. In Section \ref{heuristic}, we elaborate on the above-mentioned heuristics that there are only finitely many prime $v$-palindromes.

\section{Preliminaries}\label{sec:lemmas}
We start with some useful definitions.
\begin{defn}
Let $P$ denote any mathematical statement. Then the \emph{Iverson bracket} is defined by
\begin{equation}
    \label{eq:000}
    [ P ] \ :=\  \begin{cases}
            1,& \text{if $P$ is true}, \\
            0, &\text{if $P$ is false}.
            \end{cases}
\end{equation}
\end{defn}

\begin{defn}
For integers $\alpha\geq1$, denote $\iota(\alpha):=\alpha[\alpha>1]$. That is,
\begin{equation}
    \label{eq:000}
    \iota(\alpha) \ :=\  \begin{cases}
            0,& \text{if $\alpha\ =\ 1$}, \\
            \alpha, &\text{if $\alpha\ >\ 1$}.
            \end{cases}
\end{equation}
\end{defn}
With this notation, the additive function $v\colon\mathbb{N}\to\mathbb{Z}$ can be defined by setting in one stroke
\begin{equation}\label{v-p^a}
    v(p^\alpha) \ :=\  p + \iota(\alpha)
\end{equation}
for all prime powers $p^\alpha$.

\begin{defn}
  Let the decimal representation of an integer $n\geq1$ be\\ $(a_{L-1}\cdots a_1a_0)_{10}$, where $a_{L-1}\neq0$. Then we denote
  \begin{gather}
    a_i(n) \ :=\  a_i \quad \text{for $i\  =\  0,1,\ldots,L-1$},\\
    L(n) \ :=\  L,
  \end{gather}
  to indicate dependence on $n$.  We also by convention denote $L(0):=0$.
\end{defn}
For example,
\begin{gather}
    a_0(198) \ =\  8,\quad a_1(198) \ =\  9,\quad a_2(198) \ =\  1,\\
    L(198) \ =\  3.
  \end{gather}

Hence we have defined a function $L\colon\mathbb{N}\cup\{0\}\to\mathbb{Z}$, where $L$ stands for length. We then have the following lemmas the first two of which are obvious and follow immediately from definitions.

\begin{lemma}\label{lem:ob}
Let $0\leq m\leq n$ be integers. Then $L(m)\leq L(n)$.
\end{lemma}
\begin{lemma}\label{lem:ran}
Let $n,\ell\geq1$ be integers. Then $L(n) = \ell$ if and only if $10^{\ell-1}\leq n<10^\ell$.
\end{lemma}
\begin{lemma}\label{lem:jessie}
  Let $m,n\geq1$ be integers. Then
  \begin{equation}\label{eq:exp}
    L(mn) \ = \ L(m)+ L(n) - [mn\ <\ 10^{L(m)+L(n)-1}].
  \end{equation}
  In particular,
  \begin{equation}\label{eq:ineqL}
    L(m)+L(n)-1\ \le \  L(mn)\ \le \  L(m) + L(n).
   \end{equation}
\end{lemma}
\begin{proof}
By Lemma \ref{lem:ran}, $10^{L(m)-1}\leq m<10^{L(m)}$ and $10^{L(n)-1}\leq n<10^{L(n)}$. Therefore
\begin{equation}
  10^{L(m)+L(n)-2}\ \le \  mn\ <\ 10^{L(m)+L(n)}.
\end{equation}
If $mn<10^{L(m)+L(n)-1}$, then by Lemma \ref{lem:ran},
\begin{equation}
  L(mn) \ = \ L(m) + L(n) - 1 \ = \ L(m) + L(n) - [mn\ <\ 10^{L(m)+L(n)-1}].
\end{equation}
If $10^{L(m)+L(n)-1}\leq mn$, then by Lemma \ref{lem:ran},
\begin{equation}
  L(mn) \ = \ L(m) + L(n) \ = \ L(m) + L(n) - [mn\ <\ 10^{L(m)+L(n)-1}].
\end{equation}
This proves \eqref{eq:exp}. Because an Iverson bracket is always $0$ or $1$, clearly \eqref{eq:ineqL} holds.
\end{proof}
\begin{lemma}\label{lem:longineq}
Let $n_1,\ldots,n_k\geq1$ be integers. Then
\begin{equation}
  L(n_1\cdots n_k)\ \geq\  L(n_1) +\cdots+ L(n_k) - (k-1).
\end{equation}
\end{lemma}
\begin{proof}
This follows by repeated application of the left inequality in \eqref{eq:ineqL} in Lemma \ref{lem:jessie}.
\end{proof}
The following is an elementary inequality which essentially says that the sum is no greater than the product and which we do not prove.
\begin{lemma}\label{lem:stack}
  Let $x_1,\ldots,x_k\geq2$ be real numbers. Then
  \begin{equation}
  x_1 +\cdots+x_k\ \le \  x_1\cdots x_k.
  \end{equation}
\end{lemma}

\begin{lemma}\label{lem:preless}
Let $p$ be a prime and $\alpha\geq0$ an integer. Then
\begin{itemize}
  \item[{\rm (i)}] if $\alpha\in\{0,1\}$, then $v(p^\alpha)\leq p^\alpha\leq p+\alpha$;
  \item[{\rm (ii)}] if $\alpha>1$, then $v(p^\alpha)\ =\  p+\alpha\leq p^\alpha$.
\end{itemize}
\end{lemma}
\begin{proof}
  (i) can be easily checked. For (ii), we have, using Lemma \ref{lem:stack} twice,
  \begin{equation}
    v(p^\alpha) \ =\  p +\alpha\ \le \  p\alpha\  =\  \underbrace{p +\cdots + p}_{\alpha} \ \le \  p^\alpha.
  \end{equation}
\end{proof}

\begin{lemma}
\label{lem:ineq}
  Let $n\geq1$ be an integer. Then $v(n)\leq n$ and $L(v(n))\leq L(n)$.
\end{lemma}
\begin{proof}
We first prove that $v(n)\leq n$. We have $v(1) = 0 \leq 1$. Now assume that $n>1$ has prime factorization $n=p^{\alpha_1}_1\cdots p^{\alpha_k}_k$, where $p_1<\cdots<p_k$ are primes and $\alpha_1,\ldots,\alpha_k\geq1$ integers. Then using the fact that $v$ is additive and Lemmas \ref{lem:preless} and \ref{lem:stack},
\begin{align}
  v(n) &\ =\  v(p^{\alpha_1}_1\cdots p^{\alpha_k}_k) \ =\  v(p^{\alpha_1}_1) +\cdots+ v(p^{\alpha_k}_k)\nonumber\\
  &\ \le \  p^{\alpha_1}_1 +\cdots +  p^{\alpha_k}_k \ \le \  p^{\alpha_1}_1\cdots p^{\alpha_k}_k\  =\  n.
\end{align}
Now $L(v(n))\leq L(n)$ follows from $v(n)\leq n$ with Lemma \ref{lem:ob}.
\end{proof}

\begin{lemma}\label{lem:twoineq}
We have the following inequalities.
  \begin{itemize}
    \item[{\rm (i)}] If $x>-1$ is real, then $\log_2(10^{x+1}-1) <  (x+1)\log_210$.
    \item[{\rm (ii)}] If $n\geq2$ is an integer, then $(n+1)\log_210< 10^{n-1}$.
  \end{itemize}
\end{lemma}
\begin{proof}
\begin{itemize}
  \item[{\rm (i)}] Since $\log_2$ is strictly increasing, we have
  \begin{equation}
    \log_2(10^{x+1}-1)\  <\   \log_2(10^{x+1}) \ =\  (x+1)\log_210.
    \end{equation}
  \item[{\rm (ii)}] Define the function
  \begin{equation}
    q(n) \ =\  \frac{10^{n-1}}{n+1},\quad\text{for integers $n\geq0$}.
  \end{equation}
  Then for any $n\geq0$,
  \begin{equation}
    q(n+1) \ =\  \frac{10^n}{n+2} \ =\  \frac{10^{n-1}}{n+1}\cdot \frac{10(n+1)}{n+2} \ =\  q(n)\cdot \frac{10(n+1)}{n+2}.
  \end{equation}
  As $10(n+1)/(n+2)>1$ for $n\geq0$, we see that $q(n)$ is strictly increasing. Now because
  \begin{equation}
   q(2) \ =\  \frac{10}{3} \ =\  3.\overline{3}\ >\  3.32\cdots \ =\  \log_210,
  \end{equation}
  we see that $q(n)>\log_210$ for $n\geq2$, which is exactly what is required.
\end{itemize}
\end{proof}

\section{Setup}\label{sec:setup}

It can be checked by a computer that there are no prime $v$-palindromes of fewer than $5$ decimal digits. Therefore assume that $p$ is a prime $v$-palindrome of $m+1$ decimal digits, where $m\geq4$. In particular, according to Definition \ref{def.4}, $p\neq r(p)$ and $p=v(p) = v(r(p))$. Consequently,
\begin{equation}\label{eq:prange}
10^m+3\ \leq\  p\  \leq\  10^{m+1}-3.
\end{equation}
By the end of Section \ref{sec:otherdec}, we will have deduced that $p=5\cdot 10^m-1$ and a bit more.

In the case $r(p)<p$, by Lemma \ref{lem:ineq}, $v(r(p))\leq r(p)<p$, and thus $p = v(r(p))$ cannot hold. Therefore we may assume that $r(p)>p$. Further, in the case $r(p)$ is prime, $v(r(p))=r(p)>p$, and thus again $p = v(r(p))$ cannot hold. Therefore we may assume that $r(p)$ is composite. Suppose that
\begin{equation}\label{eq:fact}
  r(p)\ =\ fq^\beta,
\end{equation}
where $q$ is the largest prime factor of $r(p)$ and $q^\beta$ the highest power of $q$ dividing $r(p)$, namely, $q^\beta\parallel r(p)$. Let the number of decimal digits of $q$ be denoted by $\ell$, i.e., $L(q) =\ell$.

\ \\
\textbf{\emph{We shall assume the conditions and notation laid out in this section throughout the rest of this paper, without explicitly stating such assumptions in each lemma below. }} \ \\

\begin{lemma}\label{lem:basic}
We have the following:
\begin{itemize}
  \item[{\rm (i)}] $v(f)\ =\ p-v(q^{\beta})\ =\ p-q-\iota(\beta)$,
    \item[{\rm (ii)}] $L(v(f))\ =\ L(p-q-\iota(\beta))$,
    \item[{\rm (iii)}] $L(v(f))\ \geq\  L(p-q-\beta)$.
\end{itemize}
\end{lemma}
\begin{proof}
\begin{itemize}
  \item[{\rm (i)}] This follows by applying $v$ to \eqref{eq:fact}.
  \item[{\rm(ii)}] This follows by applying $L$ to part (i).
  \item[{\rm(iii)}] If $\beta>1$, then (ii) becomes $L(v(f))=L(p-q-\beta)$. If $\beta = 1$, then because $r(p)$ is composite, $f>1$, and so part (i) implies $2\leq v(f) = p-q$. Therefore
  \begin{equation}
    v(f)\ >\ p - q - 1\ \geq\ 1,
  \end{equation}
  and by Lemma \ref{lem:ob}, $L(v(f))\geq L(p-q-1)$.
\end{itemize}
\end{proof}

\begin{lemma}\label{lem:srange}
  We have $\beta\leq \log_2(10^{m+1}-1) <  10^{m-1}$.
\end{lemma}
\begin{proof}
Since $r(p)=fq^\beta$ in equation \eqref{eq:fact}, we have $r(p)\geq q^\beta$. Also, because $r(p)$ has $m+1$ decimal digits, we have $r(p)\leq 10^{m+1}-1$. Consequently,
\begin{equation}
  \beta\log q\ \le \  \log r(p)\ \le \  \log (10^{m+1}-1),
\end{equation}
and so
\begin{equation}
  \beta\ \le \  \log_q(10^{m+1}-1)\ \le \  \log_2(10^{m+1}-1).
\end{equation}
That $\log_2(10^{m+1}-1) <  10^{m-1}$ follows from Lemma \ref{lem:twoineq}, using both parts.
\end{proof}
\begin{lemma}
  \label{lem:pqs}
  If $\ell\leq m-1$, then
  \begin{itemize}
    \item[{\rm (i)}] $L(p-q-\beta)\ \geq\  m$,
    \item[{\rm (ii)}] $L(p-q-\iota(\beta))\ \geq\  m$,
    \item[{\rm (iii)}] $L(v(f))\ \geq\  m$.
  \end{itemize}
\end{lemma}
\begin{proof}
  \begin{itemize}
    \item[{\rm (i)}] Since $q$ has $\ell$ digits, $q\leq 10^{\ell}-1$. Together with \eqref{eq:prange} and Lemma \ref{lem:srange}, we have
    \begin{align}
      p-q-\beta&\ \geq\  10^m+3-10^{\ell}+1-10^{m-1}+1 \nonumber\\
      &\ \geq\  10^m+3-10^{m-1}+1-10^{m-1}+1 \nonumber\\
      &\ \geq\  8\cdot 10^{m-1}+5.
    \end{align}
    Therefore $p-q-\beta$ has at least $m$ decimal digits.
    \item[{\rm (ii)}] This is because $p-q-\iota(\beta)\geq p-q-\beta$.
    \item[{\rm (iii)}] This is by combining Lemma \ref{lem:basic}(iii) and part (i).
  \end{itemize}
\end{proof}

\begin{lemma}
\label{lem:ob2}
  We have
  \begin{itemize}
    \item[{\rm (i)}]
    $L(f) \ =\  m+1-L(q^{\beta})+[r(p)\ <\ 10^{L(f)+L(q^{\beta})-1}]$,
    \item[{\rm (ii)}]  $L(q^{\beta}) \ =\  m+1-L(f)+[r(p)\ <\ 10^{L(f)+L(q^\beta)-1}]$,
\item[{\rm (iii)}] $L(q^{\beta})\ \geq\  \ell$,
\item[{\rm (iv)}] $L(f)\ \leq\  m+1$,
\item[{\rm (v)}] $L(f)\ \leq\  m+2-\ell$,
\item[{\rm (vi)}] $L(f)\ \leq\  m+1-\beta(\ell-1)$, and 
\item[{\rm (vii)}] $L(p-q-\beta)\ \leq\  m+2-\ell$.
  \end{itemize}
\end{lemma}
\begin{proof}
\begin{itemize}
  \item[{\rm (i)}] By \eqref{eq:fact} and Lemma \ref{lem:jessie}, we have
  \begin{align}
    m+1&\ =\ L(r(p)) \ =\  L(f q^\beta) \ =\  L(f)+L(q^\beta)-[fq^\beta\ <\ 10^{L(f)+L(q^\beta)-1}]\nonumber\\
    & \ =\  L(f)+L(q^\beta)-[r(p)\ <\ 10^{L(f)+L(q^\beta)-1}].
  \end{align}
  The required equality then follows by rearranging.
  \item[{\rm (ii)}] This follows by rearranging part (i).
  \item[{\rm (iii)}] Since $q^\beta\geq q$, by Lemma \ref{lem:ob} we have $L(q^\beta)\geq L(q)=\ell$.
  \item[{\rm (iv)}] By \eqref{eq:fact}, we have $f\leq r(p)$. Thus by Lemma \ref{lem:ob}, we have $L(f)\leq L(r(p)) = m+1$.
  \item[{\rm (v)}] By parts (i) and (iii) and the fact that an Iverson bracket must be no greater than $1$, we have
  \begin{align}
    L(f) &\ =\  m+1-L(q^{\beta})+[r(p)\ <\ 10^{L(f)+L(q^{\beta})-1}] \nonumber\\
    & \ \le \  m+1-\ell+1 \ =\  m+2-\ell.
  \end{align}
  \item[{\rm (vi)}] By Lemma \ref{lem:longineq}, we have
  \begin{equation}
    L(q^\beta)\  \geq\  \beta L(q)-(\beta-1) \ =\  \beta \ell - (\beta - 1) \ =\  \beta(\ell-1)+1.
    \end{equation}
    Consequently from part (i),
    \begin{align}
      L(f) &\ =\  m+1-L(q^{\beta})+[r(p)\ <\ 10^{L(f)+L(q^{\beta})-1}]\nonumber\\
      &\ \le \  m+1-\beta(\ell-1)-1+1 \ =\  m+1-\beta(\ell-1).
    \end{align}
    \item[{\rm (vii)}] By Lemma \ref{lem:basic}(iii), Lemma \ref{lem:ineq}, and part (v),
    \begin{equation}
      L(p-q-\beta)\ \le \  L(v(f))\ \le \  L(f)\ \le \  m+2-\ell.
    \end{equation}
\end{itemize}
\end{proof}

\section{The case $\ell\leq m$}\label{sec:leqm}

Since $r(p) = fq^\beta$ and the number of decimal digits of $r(p)$ and $q$ are $m+1$ and $\ell$, respectively, clearly $m+1\geq \ell$. In this section we consider the case $\ell\leq m$, dividing it into four cases corresponding to the four subsections below, and in each case show that a contradiction results. This means that necessarily $\ell = m+1$, which we consider in the next section.

\ \\
\subsection{Case $\ell=1$} Since $\ell=1\leq m-1$, by Lemma \ref{lem:pqs}(iii), $L(v(f))\geq m$. By Lemma \ref{lem:ineq} and Lemma \ref{lem:ob2}(iv),
\begin{equation}
  m\ \le \  L(v(f))\ \le \  L(f)\ \le \  m+1.
\end{equation}
By Lemma \ref{lem:ob2}(ii),
\begin{equation}
  L(q^{\beta}) \ =\  m+1-L(f)+[r(p)\ <\ 10^{L(f)+L(q^{\beta})-1}]\ \le \  1+[r(p)\ <\ 10^{L(f)+L(q^{\beta})-1}]\ \le \  2.
\end{equation}
So we have $L(q^{\beta})\leq 2$. If $q=2$, then because $q$ is the largest prime factor of $r(p)$, necessarily $f=1$. This implies that $r(p)=2^{\beta}$ has at most $2$ decimal digits, which is a contradiction. Hence $q\in\{3,5,7\}$. There remains only $8$ possibilities for $q^{\beta}$ and by checking one by one, it can be seen that $3\leq v(q^{\beta})\leq 9$. By Lemma \ref{lem:basic}(i) and \eqref{eq:prange},
\begin{equation}
10^m-6\ =\ 10^m+3-9\ \le \  v(f)\ =\ p-v(q^\beta)\ \le \  10^{m+1}-3-3 \ =\  10^{m+1}-6.
\end{equation}
Thus we have
\begin{equation}\label{eq:cons}
10^m-6 \ \le \   v(f)\ \le \  10^{m+1}-6.
\end{equation}
In the remainder of this subsection we discuss the cases $q=3$, $q=5$, and $q=7$, one by one, showing that each case leads to a contradiction. This means that the whole case $\ell=1$ leads to a contradiction.

\ \\
\underline{Sub case $q=3$}: We must have $r(p) = 2^{\gamma} 3^{\beta}$, where $\gamma\geq0$ is an integer. Since $L(3^\beta)\leq2$, we have $3^\beta\leq 81$. Therefore $10^4\leq r(p) \leq 2^\gamma \cdot 81$, and so $\gamma\geq7$. Thus because $f=2^{\gamma}$, \eqref{eq:cons} and Lemma \ref{lem:preless} implies
\begin{equation}
10^m-6\ \le \  v(f) \ = \ v(2^\gamma)\ \le \  2+\gamma.
\end{equation}
Consequently,
\begin{equation}
10^m-8 \ \le \   \gamma.
\end{equation}
Hence
\begin{equation}\label{eq:just}
  r(p) \ =\  2^{\gamma} 3^{\beta}\ \geq\  2^{10^m-8}\cdot 3\ \geq\  10^{m+1}
\end{equation}
(the last inequality can be shown to hold for $m\geq2$). This contradicts the fact that $L(r(p)) = m+1$.

\ \\
\underline{Sub case $q=5$}: We must have $r(p) = 2^{\delta}3^{\gamma}5^{\beta}$, where $\gamma,\delta\geq0$ are integers.  Thus because $f=2^{\delta}3^{\gamma}$, by \eqref{eq:cons} and Lemma \ref{lem:preless},
\begin{equation}\label{eq:tov}
10^m-6\ \le \   v(f) \ = \ v(2^\delta)+ v(3^\gamma) \ \le \  2+\delta+3+\gamma.
\end{equation}
Consequently,
\begin{equation}
\delta+\gamma\ \geq\  10^m-11.
\end{equation}
Hence
\begin{equation}
  r(p) \ = \ 2^{\delta}3^{\gamma}5^{\beta}\ \geq\  2^{\delta+\gamma}\cdot 5\ \geq\  2^{10^m-11}\cdot 5\ \geq\  10^{m+1}
\end{equation}
(the last inequality can be shown to hold for $m\geq2$). This contradicts the fact that $L(r(p)) = m+1$.

\ \\
\underline{Sub case $q=7$}: We must have $r(p) = 2^{\eta}3^{\delta}5^{\gamma}7^{\beta}$, where $\gamma, \delta, \eta \geq0$ are integers. Thus because $f=2^{\eta}3^{\delta}5^{\gamma}$, by \eqref{eq:cons} and Lemma \ref{lem:preless},
\begin{equation}\label{eq:tovs}
10^m-6\ \le \   v(f) \ = \ v(2^\eta)+v(3^\delta)+v(5^\gamma) \ \le \  10+\eta+\delta+\gamma.
\end{equation}
Consequently,
\begin{equation}
  \eta+\delta+\gamma\ \geq\  10^m-16.
\end{equation}
We then have
\begin{equation}
r(p)\ \geq\  2^{\eta+\delta+\gamma}\cdot 7\ \geq\  2^{10^m-16}\cdot 7\ \geq\  10^{m+1}
\end{equation}
(the last inequality can be shown to hold for $m\geq2$). This contradicts the fact that $L(r(p)) = m+1$.

\ \\

\subsection{Case $\ell=2$} By Lemma \ref{lem:ob2}(iii), $L(q^{\beta})\geq 2$, and by Lemma \ref{lem:ob2}(v), $L(f)\leq m$. Since $\ell=2\leq m-1$, by Lemma \ref{lem:pqs}(iii), $L(v(f))\geq m$. By Lemma \ref{lem:ineq},
\begin{equation}
  m\ \le \  L(v(f))\ \le \  L(f)\ \le \  m.
\end{equation}
Therefore $L(f) = L(v(f)) = m$. By Lemma \ref{lem:ob2}(vi),
\begin{equation}
  m\ =\ L(f)\ \le \  m+1-\beta,
\end{equation}
and thus $\beta=1$. Hence \eqref{eq:fact} simplifies to $r(p)=fq$. Since $\ell=2$, we have $11\leq q\leq 97$. Therefore because $L(r(p)) = m+1$,
\begin{equation}
  f\ =\ \frac{r(p)}{q}\ <\  \frac{10^{m+1}}{11}\ \le \  10^m-100
\end{equation}
(it can be shown that the rightmost inequality holds for $m\geq4$). By taking the $v$ of $r(p) = fq$, we have
\begin{equation}
  p \ =\  v(f) + q \ \le \  f+97\ <\ 10^m-100 + 97\ =\ 10^m-3.
\end{equation}
This implies that $L(p)<m+1$, which is a contradiction.

\ \\

\subsection{Case $3\leq \ell\leq m-1$ }Since $\ell\leq m-1$, by Lemma \ref{lem:pqs}(i) and Lemma \ref{lem:ob2}(vii),
\begin{equation}
  m\ \le \  L(p-q-\beta)\ \le \  m+2-\ell.
\end{equation}
This implies that $\ell\leq 2$, a contradiction. Hence this case is impossible.

\ \\
\subsection{Case $\ell=m$} By Lemma \ref{lem:ob2}(iii), $L(q^{\beta})\geq m$, and by Lemma \ref{lem:ob2}(v), $L(f)\leq 2$.
Therefore $L(f)\in\{1,2\}$. By Lemma \ref{lem:ob2}(vi),
\begin{equation}
  1\ \le \  L(f)\ \le \  m+1-\beta(m-1).
\end{equation}
This implies that
\begin{equation}
  \beta\ \le \ \frac{m}{m-1}\ =\ 1+\frac{1}{m-1},
\end{equation}
and so $\beta=1$. Therefore $r(p) = fq$. If $q=2$, then $f = 1$ and so $r(p)=2$, which contradicts $L(r(p)) = m+1\geq5$. Hence $q$ is an odd prime. In addition, if $f=1$, then $r(p)=q$ is prime, contrary to our assumption that $r(p)$ is composite. Hence $f>1$. By Lemma \ref{lem:basic}(i), $v(f) = p-q$, and so $v(f)$ is even. In the following we consider the cases $L(f) = 1$ and $L(f) = 2$ separately, showing that each leads to a contradiction and so ultimately this case $\ell=m$ is also impossible.

\ \\
\underline{Sub case $L(f)=1$}: Since $v(f)$ is even and $2\leq f\leq 9$, we have $f\in\{2,4\}$. In the case $f=2$, we have $2=p-q$. Then by \eqref{eq:prange},
\begin{equation}
q\ =\ p-2\ \geq\  10^m+3-2\ =\ 10^m+1,
\end{equation}
contradicting that $L(q) = m$. In the case $f=4$, we have $4=p-q$. Similarly by \eqref{eq:prange},
\begin{equation}
q\ =\ p-4\ \geq\  10^m+3-4\ =\ 10^m-1.
\end{equation}
As $L(q) = m$, we have $q=10^m-1$, contradicting the primeness of $q$.

\ \\
\underline{Sub case $L(f)=2$}: Since $v(f)$ is even and $10 \leq f\leq 99$, we see that $f$ must be one of
\begin{align}
&15,\
16,\
21,\
24,\
27,\
30,\
33,\
35,\
39,\
40,\
42,\
45,\
51,\
54,\
55,\
56,\
57,\
60,\
63,\nonumber\\
&64,65,\
66,\
69,70, 72,\
75,\
77,\
78,\
84,\
85,\
87,\
88,\
90,\
91,\
93,\
95,\
96,\
99,
\end{align}
with $v(f)$ being one of
\begin{align}
6,\ 8,\ 10,\ 12,\ 14,\ 16,\ 18,\ 20,\ 22,\ 24,\ 26,\ 32,\ 34.
\end{align}
Consequently, by \eqref{eq:prange},
\begin{equation}
  q\ =\ p-v(f)\ \geq\  10^m+3 -34 \ =\  10^m-31,
\end{equation}
and so
\begin{equation}
  r(p) \ =\  fq \ \geq\  15 (10^m-31)\ \geq\  10^{m+1}
\end{equation}
(it can be shown that the rightmost inequality holds for $m\geq2$). This contradicts the fact that $L(r(p)) = m+1$.

\section{The case $\ell\ =\ m+1$}\label{sec:mp1}
In this section we consider the case $\ell=m+1$ and narrow down the potentially possible values of $p$ more, i.e., deduce more necessary conditions.

By Lemma \ref{lem:ob2}(v), $L(f)=1$.
By Lemma \ref{lem:ob2}(vi),
\begin{equation}
  1 \ =\  L(f)\ \le \  m+1-\beta m.
\end{equation}
This implies that $\beta\leq 1$, and so $\beta=1$. Therefore $r(p) = fq$. If $q=2$, then $f = 1$ and so $r(p)=2$, which contradicts $L(r(p)) = m+1\geq5$. Hence $q$ is an odd prime. In addition, if $f=1$, then $r(p)=q$ is prime, contrary to our assumption that $r(p)$ is composite. Hence $f>1$. By Lemma \ref{lem:basic}(i), $v(f) = p-q$, and so $v(f)$ is even. As $L(f) = 1$, we see that $v(f) = f\in\{2,4\}$. Consequently, $r(p)$ must be even.

Let the decimal representations of $p$, $r(p)$, and $q$ be
\begin{align}
  p &\ =\  (a_m\cdots a_0)_{10},\label{eq:reprp}\\
  r(p) &\ =\   (a_0\cdots a_m)_{10},\label{eq:reprrp}\\
  q & \ =\  (b_m\cdots b_0)_{10}, \label{eq:repq}
\end{align}
where $a_m,b_m\neq0$.
As $p$ is odd and prime, $a_0\in\{1,3,7,9\}$, and as $r(p)$ is even, $a_m\in\{2,4,6,8\}$. Since $v(f) = p-q$, we have $q = p-v(f)$, and so
\begin{equation}\label{eq:long}
  (b_m\cdots b_0)_{10} \ =\  (a_m\cdots a_0)_{10}-v(f).
\end{equation}
Consequently, because $v(f)\in\{2,4\}$,
\begin{equation}
  a_m\cdot 10^m-4 \ \le \  (b_m\cdots b_0)_{10}\ <\  (a_m\cdots a_0)_{10},
\end{equation}
and so $b_m\in\{a_m-1,a_m\}$. Since $r(p) = fq$,
\begin{equation}\label{eq:rfq}
  (a_0\cdots a_m)_{10} \ =\ f(b_m\cdots b_0)_{10}.
\end{equation}
This implies that
\begin{equation}\label{eq:forcon}
  fb_m\ \le \  a_0.
  \end{equation}
  Hence
\begin{equation}\label{eq:erange}
  a_m-1\ \le \  b_m\ \le \  \frac{a_0}{f},\quad\text{which implies that}\quad a_m\ \le \ 1+\frac{a_0}{f}\ \le \  1+\frac{9}{f}.
\end{equation}
Notice that from \eqref{eq:long} we have
\begin{equation}\label{eq:mod1}
  b_0\ \equiv\  a_0-v(f)\pmod{10},
\end{equation}
and that from \eqref{eq:rfq} we have
\begin{equation}\label{eq:mod2}
  a_m\  \equiv\  fb_0 \pmod{10}.
\end{equation}
In the following we consider the cases $f=2$ and $f=4$ separately, corresponding to two subsections. For $f=2$, we show that necessarily $a_m=4$, $a_0=9$, $b_m=4$, and $b_0=7$; while for $f=4$, we show that a contradiction results.

\ \\
\subsection{Case $f=2$}\eqref{eq:mod1} and \eqref{eq:mod2} become respectively
\begin{eqnarray}
  b_0 & \ \equiv\ & a_0-2\pmod{10}, \label{eq:mod12}\\
  a_m & \ \equiv\ & 2b_0 \pmod{10} \label{eq:mod22}.
\end{eqnarray}
By \eqref{eq:erange}, $a_m\leq 1+9/2 = 5.5$ and so as $a_m\in\{2,4,6,8\}$, we have $a_m\in\{2,4\}$. In the following we consider the cases $a_m=2$ and $a_m=4$ separately. For $a_m=2$, we show that a contradiction results; while for $a_m=4$, we show that necessarily $a_0=9$, $b_m=4$, and $b_0=7$.

\ \\
\underline{Sub case $a_m=2$}: \eqref{eq:mod22} becomes $2\equiv 2b_0 \pmod{10}$, or equivalently, $b_0\equiv 1\pmod{5}$. Thus as $0\leq b_0<10$, we have $b_0\in\{1,6\}$. By \eqref{eq:mod12}, we have modulo $10$,
\begin{equation}
  a_0\ \equiv\  b_0+2\ \equiv\ \begin{cases}
            3,& \text{if $b_0\ =\ 1$}, \\
            8, &\text{if $b_0\ =\ 6$}.
            \end{cases}
\end{equation}
Thus as $0\leq a_0<10$,
\begin{equation}
  a_0\ =\ \begin{cases}
            3,& \text{if $b_0\ =\ 1$}, \\
            8, &\text{if $b_0\ =\ 6$}.
            \end{cases}
\end{equation}
As $a_0\in\{1,3,7,9\}$, we have $b_0=1$ and $a_0=3$. Consequently, \eqref{eq:long} becomes
\begin{equation}
  (b_m\cdots 1)_{10} \ =\  (2\cdots 3)_{10}-2,
\end{equation}
which means that $b_m=2$. Then however, \eqref{eq:forcon} becomes $2\cdot 2\leq 3$, which is false and we have a contradiction.

\ \\
\underline{Sub case $a_m=4$}: \eqref{eq:mod22} becomes $4\equiv 2b_0 \pmod{10}$, or equivalently, $b_0\equiv 2\pmod{5}$. Thus as $0\leq b_0<10$, we have $b_0\in\{2,7\}$. By \eqref{eq:mod12}, we have modulo $10$,
\begin{equation}
  a_0\ \equiv\  b_0+2\ \equiv\ \begin{cases}
            4,& \text{if $b_0\ =\ 2$}, \\
            9, &\text{if $b_0\ =\ 7$}.
            \end{cases}
\end{equation}
Thus as $0\leq a_0<10$,
\begin{equation}
  a_0\ =\ \begin{cases}
            4,& \text{if $b_0\ =\ 2$}, \\
            9, &\text{if $b_0\ =\ 7$}.
            \end{cases}
\end{equation}
As $a_0\in\{1,3,7,9\}$, we have $b_0=7$ and $a_0=9$. Consequently, \eqref{eq:long} becomes
\begin{equation}
  (b_m\cdots 7)_{10} \ =\  (4\cdots 9)_{10}-2,
\end{equation}
which means that $b_m=4$. Then \eqref{eq:forcon} becomes $2\cdot 4\leq 9$, which is true and so we do not have a contradiction like we just did in the case $a_m=2$.

\ \\
\subsection{Case $f=4$} \eqref{eq:mod1} and \eqref{eq:mod2} become respectively
\begin{eqnarray}
  b_0 & \ \equiv\ & a_0-4\pmod{10}, \label{eq:mod14}\\
  a_m & \ \equiv \ & 4b_0 \pmod{10} \label{eq:mod24}.
\end{eqnarray}
By \eqref{eq:erange}, $a_m\leq 1+9/4 = 3.25$ and so as $a_m\in\{2,4,6,8\}$, we have $a_m=2$.

Thus \eqref{eq:mod24} becomes $2\equiv 4b_0 \pmod{10}$, or equivalently, $2b_0\equiv 1\pmod{5}$, or equivalently, $b_0\equiv 3\pmod{5}$. Thus as $0\leq b_0<10$, we have $b_0\in\{3,8\}$. By \eqref{eq:mod14}, we have modulo $10$,
\begin{equation}
  a_0\ \equiv\  b_0+4\ \equiv\ \begin{cases}
            7,& \text{if $b_0\ =\ 3$}, \\
            12, &\text{if $b_0\ =\ 8$}.
            \end{cases}
\end{equation}
Thus as $0\leq a_0<10$,
\begin{equation}
  a_0\ =\ \begin{cases}
            7,& \text{if $b_0\ =\ 3$}, \\
            2, &\text{if $b_0\ =\ 8$}.
            \end{cases}
\end{equation}
As $a_0\in\{1,3,7,9\}$, we have $b_0=3$ and $a_0=7$. Consequently, \eqref{eq:long} becomes
\begin{equation}
  (b_m\cdots 3)_{10} \ =\  (2\cdots 7)_{10}-4,
\end{equation}
which means that $b_m=2$. Then however, \eqref{eq:forcon} becomes $4\cdot 2\leq 7$, which is false and we have a contradiction.

\section{The other decimal digits of $p$}\label{sec:otherdec}
As a result of Sections \ref{sec:setup} through \ref{sec:mp1}, we see that if $p$ is a prime $v$-palindrome, then the following are true:
\begin{itemize}
  \item[{\rm (i)}] $p$ has $m+1$ decimal digits, for some $m\geq4$,
  \item[{\rm (ii)}] $p$ is of the form $(4\cdots 9)_{10}$, i.e., its leftmost decimal digit is $4$ and its rightmost decimal digit is $9$,
  \item[{\rm (iii)}] $p-2$ is prime, and
  \item[{\rm (iv)}] $r(p) = 2(p-2)$.
\end{itemize}
In this section we show further that all other decimal digits of $p$ must be $9$'s as well, and thus $p = 5\cdot 10^m - 1$. Filling in what we know into \eqref{eq:reprp}, \eqref{eq:reprrp}, and \eqref{eq:repq}, we have
\begin{align}
  p &\ =\  (4,a_{m-1},\ldots,a_1,9)_{10},\label{eq:reprpN}\\
  r(p) &\ =\   (9,a_1,\ldots,a_{m-1},4)_{10},\label{eq:reprrpN}\\
  p-2 & \ =\  (4,a_{m-1},\ldots ,a_1,7)_{10}. \label{eq:repqN}
\end{align}
Since $r(p) = 2(p-2)$,
\begin{equation}\label{eq:rfqN}
  (9,a_1,\ldots,a_{m-1},4)_{10} \ =\ 2(4,a_{m-1},\ldots ,a_1,7)_{10}.
\end{equation}
We need to prove that
\begin{equation}\label{eq:need}
  a_i \ =\  a_{m-i} \ =\  9,\quad \text{for $1\ \leq\  i\ \leq\  \left\lfloor\frac{m}{2}\right\rfloor$}.
\end{equation}

For integers $0\leq I\leq \left\lfloor\frac{m}{2}\right\rfloor$, let $S(I)$ be the statement that
\begin{equation}
  a_i \ =\  a_{m-i} \ =\  9,\quad \text{for $1\ \leq\  i\ \leq\  I$}.
\end{equation}
We prove that $S(I)$ holds for all $0\leq I\leq \left\lfloor\frac{m}{2}\right\rfloor$ inductively, which will imply in particular that $S(\left\lfloor\frac{m}{2}\right\rfloor)$, i.e., \eqref{eq:need}, holds. Firstly, notice that $S(0)$ holds vacuously. Next, suppose that $S(I)$ holds for some $0\leq I< \left\lfloor\frac{m}{2}\right\rfloor$. We shall proceed to prove $S(I+1)$, which amounts to proving
\begin{equation}
  a_{I+1} \ =\  a_{m-I-1} \ =\  9.
\end{equation}
We have
\begin{align}
  p &\ =\  (4,\{9\}^{I},a_{m-I-1},\ldots,a_{I+1},\{9\}^{I+1})_{10},\\
  r(p) &\ =\  (\{9\}^{I+1},a_{I+1},\ldots,a_{m-I-1},\{9\}^{I},4)_{10},\\
  p-2 &\ =\  (4,\{9\}^{I},a_{m-I-1},\ldots,a_{I+1},\{9\}^{I},7)_{10},
\end{align}
where $\{9\}^j$ for some integer $j\geq0$ means that there are $j$ digits of $9$ consecutively; $\{9\}^0$ means that there is nothing. \eqref{eq:rfqN}
becomes
\begin{equation}\label{eq:crucial}
  (\{9\}^{I+1},a_{I+1},\ldots,a_{m-I-1},\{9\}^{I},4)_{10} \ =\ 2(4,\{9\}^{I},a_{m-I-1},\ldots,a_{I+1},\{9\}^{I},7)_{10}.
\end{equation}
If $a_{m-I-1}\leq 4$, then the right-hand side of \eqref{eq:crucial} must be of the form
\begin{equation}
  (\underbrace{9,\ldots,9}_{I},8,\ldots)_{10},
\end{equation}
which cannot equal the left-hand side. Therefore necessarily $5\leq a_{m-I-1}\leq 9$. Notice that the congruence
\begin{equation}\label{eq:crucialcong}
  2a_{I+1}+1\ \equiv\  a_{m-I-1}\pmod{10}
\end{equation}
follows from \eqref{eq:crucial}. Reducing this congruence to modulo $2$, we see that $a_{m-I-1}$ must be odd. Therefore necessarily $a_{m-I-1}\in\{5,7,9\}$. In the following we consider each such possible value of $a_{m-I-1}$, corresponding to three subsections.

\ \\
\subsection{Case $a_{m-I-1} = 5$} \eqref{eq:crucial} becomes
\begin{equation}\label{eq:crucial5}
  (\{9\}^{I+1},a_{I+1},\ldots,5,\{9\}^{I},4)_{10} \ =\ 2(4,\{9\}^{I},5,\ldots,a_{I+1},\{9\}^{I},7)_{10}
\end{equation}
and \eqref{eq:crucialcong} becomes
\begin{equation}
  2a_{I+1} + 1\ \equiv\  5\pmod{10},
\end{equation}
which forces $a_{I+1} \in \{2,7\}$. However, in view of integer multiplication, we see that the digit of $10^{m-I-1}$ of the right-hand side of \eqref{eq:crucial5} must be $0$ or $1$. This means that we need to have $a_{I+1}\in\{0,1\}$, which is impossible.

\ \\
\subsection{Case $a_{m-I-1} = 7$} \eqref{eq:crucial} becomes
\begin{equation}\label{eq:crucial7}
  (\{9\}^{I+1},a_{I+1},\ldots,7,\{9\}^{I},4)_{10} \ =\ 2(4,\{9\}^{I},7,\ldots,a_{I+1},\{9\}^{I},7)_{10}
\end{equation}
and \eqref{eq:crucialcong} becomes
\begin{equation}
  2a_{I+1} + 1\ \equiv\  7\pmod{10},
\end{equation}
which forces $a_{I+1} \in \{3,8\}$. However, in view of integer multiplication, we see that the digit of $10^{m-I-1}$ of the right-hand side of \eqref{eq:crucial7} must be $4$ or $5$. This means that we need to have $a_{I+1}\in\{4,5\}$, which is impossible.

\ \\
\subsection{Case $a_{m-I-1} = 9$}  \eqref{eq:crucial} becomes
\begin{equation}\label{eq:crucial9}
  (\{9\}^{I+1},a_{I+1},\ldots,\{9\}^{I+1},4)_{10} \ =\ 2(4,\{9\}^{I+1},\ldots,a_{I+1},\{9\}^{I},7)_{10}
\end{equation}
and \eqref{eq:crucialcong} becomes
\begin{equation}
  2a_{I+1}+ 1\ \equiv\  9\pmod{10},
\end{equation}
which forces $a_{I+1} \in \{4,9\}$.

Assume that $a_{I+1} = 4$, then \eqref{eq:crucial9} becomes
\begin{equation}\label{eq:crucial94}
  (\{9\}^{I+1},4,\ldots,\{9\}^{I+1},4)_{10} \ =\ 2(4,\{9\}^{I+1},\ldots,4,\{9\}^{I},7)_{10}.
\end{equation}
However, in view of integer multiplication, we see that the digit of $10^{m-I-1}$ of the right-hand side of \eqref{eq:crucial94} must be $8$ or $9$, in contrary to the left-hand side.  Hence we must have $a_{I+1}= 9$.

Notice that this completes the induction because we are in the final case of $a_{m-I-1} = 9$.

\section{Proof of the converse}\label{sec:converse}

Sections \ref{sec:setup} through \ref{sec:otherdec} proved the forward direction of Theorem \ref{thm:main}. In this section we prove the converse.

Let $p = 5\cdot 10^m - 1 = 4\underbrace{9\cdots 9}_{m}$, for some integer $m\geq4$, be a prime such that $p-2 = 5\cdot 10^m - 3$ is also prime. We show that $p$ is a $v$-palindrome. Firstly, clearly $10\nmid p$ and $p\neq r(p)$. We have
\begin{align}
  r(p) \ =\  r(4\underbrace{9\cdots 9}_{m}) \ =\  \underbrace{9\cdots 9}_{m}4 \ =\  2\cdot 4\underbrace{9\cdots 9}_{m-1}7 \ =\  2(p-2).
\end{align}
Consequently, as $p-2$ is an odd prime,
\begin{equation}
  v(r(p)) \ =\  v(2(p-2)) \ =\  2+(p-2) \ =\  p.
\end{equation}
This completes the proof.

\section{Number Of Prime $v$-palindromes}\label{heuristic}

Theorem \ref{thm:heuristicfinite} follows from standard models for prime numbers; we sketch below how a slightly weakened Cram\'er model, combined with our characterization of the form of prime $v$-palindromes, implies that there can only be finitely many.

For the standard Cram\'er model, one assumes that each integer $n$ is prime with probability on the order of $1/\log n$, and the probability any two numbers are both prime is simply the product of the probabilities. This of course is clearly false, as we know if $n \ge 2$ is even then it cannot be prime, and if $n \equiv -2 \bmod p$ for any prime $p < n$ then $n+2$ cannot be prime. However, our goal is simply to provide support, and thus we ignore the more refined arguments one can do (see for example \cite{Rub}). We assume instead that the probability $n$ and $n+2$ are both prime is bounded by $C / \log^2 n$ for some fixed $C$; as we are only trying to prove there are at most finitely many prime $v$-palindromes, we are fine with a slightly larger but still finite upper bound.

Let $T_n$ be the event that $5 \cdot 10^n - 3$ and $5 \cdot 10^n - 1$ are both prime, then the expected number of prime $v$-palindromes at most $10^{N+1}$ is \begin{equation} \sum_{n=1}^N 1 \cdot {\rm Prob}(T_n). \end{equation} As we are just concerned with supporting the conjecture that there are only finitely many, let us over-estimate and say
\begin{equation}\label{eq:heuristicsum} {\rm Prob}(T_n) \ \le \ \frac{C}{\log^2(5 \cdot 10^n - 3)} \ \le \  \frac{400C}{n^2 \log^2 10} \ \le
\ \frac{100C}{n^2}. \end{equation} As the sum of $1/n^2$ converges, the expected number of prime $v$-palindromes is finite.

\begin{rek} The Cram\'er model suggests we can take $C$ to be around 1. With such an assumption, given that there are no prime $v$-palindromes for the first several candidates of the form $5 \cdot 10^m - 1$, the expected number of numbers of this form that are the larger in a twin prime pair is less than 1/2, and thus we do not expect there to be any prime $v$-palindromes. \end{rek}

\begin{rek} While standard models predict the probability two integers of size $x$ differing by 2 are both prime is on the order of $1/\log^2 x$, a significantly larger bound would still imply there are only finitely many $v$-primes. For example, if we instead had the probability bounded by a quantity of size $1/\log^{1+\epsilon} x$ for any $\epsilon > 0$ we would still get a finite sum in \eqref{eq:heuristicsum}. \end{rek} 


\ \\

\end{document}